\documentclass[12pt]{amsart}

\usepackage[english]{babel}
\usepackage[latin9]{inputenc}

\usepackage{geometry}
\usepackage[usenames,dvipsnames]{xcolor}
\usepackage{mathtools}
\usepackage{graphicx}
\usepackage[numbers]{natbib}

\usepackage[pdfpagelabels,
	bookmarksopen,
	pdfstartview={FitV},
	plainpages=false,
	breaklinks=true]
	{hyperref}

\newcommand{\R}{\ensuremath{\mathbb R}}
\newcommand{\N}{\ensuremath{\mathbb N}}

\newcommand{\abs}[1]{\ensuremath{ \left\vert #1 \right\vert } }

\newcommand{\eps}{\ensuremath{\varepsilon}}
\newcommand{\ndist}[2]{\ensuremath{\mathcal{N}(#1,#2)}}

\newcommand{\eqdist}{\ensuremath{\stackrel{\mathrm{d}}{=}}}
\newcommand{\cond}{\mid}
\newcommand{\prob}{\operatorname{P}}
\newcommand{\E}[1]{\operatorname{E}\left(#1\right)}

\newcommand{\ds}{\ensuremath{ \,\mathrm{d}s } }
\newcommand{\dt}{\ensuremath{ \,\mathrm{d}t } }
\newcommand{\dx}{\ensuremath{ \,\mathrm{d}x } }

\theoremstyle{plain}
	\newtheorem{lem}{Lemma}
	\newtheorem{thm}{Theorem}
	\newtheorem{cor}{Corollary}
	\newtheorem{prop}{Proposition}
\theoremstyle{definition}

\theoremstyle{remark}
	\newtheorem{rem}{Remark}

\begin{document}

\title[Adaptive Approximation of the Brownian Minimum] 
	{Adaptive Approximation of the Minimum\\   
	of Brownian Motion}

\author[Calvin]
{James M. Calvin}
\address{Department of Computer Science\\
	New Jersey Institute of Technology\\
	Newark, NJ 07102-1982\\
	USA}
\email{calvin@njit.edu}

\author[Hefter]
{Mario Hefter}
\address{Department of Mathematics\\
	Technische Universit\"at Kaisers\-lautern\\
	67653 Kaisers\-lautern\\
	Germany}
\email{hefter@mathematik.uni-kl.de}

\author[Herzwurm]
{Andr\'{e} Herzwurm}
\address{Department of Mathematics\\
	Technische Universit\"at Kaisers\-lautern\\
	67653 Kaisers\-lautern\\
	Germany}
\email{herzwurm@mathematik.uni-kl.de}

%-------------------------------------------------------------------
%-------------------------------------------------------------------
% Abstract
%-------------------------------------------------------------------
%-------------------------------------------------------------------
\begin{abstract}
We study the error in approximating the minimum of a Brownian motion
on the unit interval based on finitely many point evaluations.
We construct an algorithm that adaptively chooses
the points at which to evaluate the Brownian path.
In contrast to the $1/2$ convergence rate of optimal nonadaptive algorithms,
the proposed adaptive algorithm converges at an arbitrarily high polynomial rate.
\end{abstract}


\maketitle

%-------------------------------------------------------------------
%-------------------------------------------------------------------
% Introduction
%-------------------------------------------------------------------
%-------------------------------------------------------------------
\section{Introduction}\label{sec:intro}
We study the pathwise approximation of the minimum
\begin{align*}
	M = \inf_{0\leq t\leq 1} W(t)
\end{align*}
of a Brownian motion $W$ on the unit interval $[0,1]$ based on
adaptively chosen function values of $W$. In contrast to nonadaptive
algorithms, which evaluate a function always at the same points,
adaptive algorithms may sequentially choose points at which to
evaluate the function. For the present problem, this means that
the $n$-th evaluation site may depend on the first $n-1$ observed values
of the Brownian path $W$. Given a number of evaluation sites,
we are interested in algorithms that have a small error in the residual sense
with respect to the $L_p$-norm.

A key motivation for studying this approximation problem stems from numerics
for the reflected Brownian motion given by
\begin{align*}
	\hat{W}(t) = W(t) - \inf_{0\le s\le t}W(s).
\end{align*}
Apart from its use in queueing theory \cite{harrison}, the reflected
Brownian motion also appears in the context of nonlinear stochastic differential
equations. More precisely, the solution process of a particular instance of
a Cox-Ingersoll-Ross process is given by the square of $\hat{W}$.
Hence numerical methods for the approximation of $M$ can be used for the
approximation of $\hat{W}$ and thus for the corresponding Cox-Ingersoll-Ross process.
We refer to~\cite{2015-bessel} for such an application of the algorithm
proposed in this paper.

The complexity analysis of pathwise approximation of the Brownian minimum $M$ based
on finitely many function evaluations was initiated in \cite{ritter1990},
where it was shown that for any nonadaptive algorithm
using $n$ function evaluations the average error is at least of order $n^{-1/2}$.
Moreover, a simple equidistant discretisation already
has an error of order $n^{-1/2}$, and thus achieves
the lower bound for nonadaptive algorithms.
A detailed analysis of the asymptotics of the pathwise error
in case of an equidistant discretisation was undertaken in \cite{Asmussen}.

The situation regarding adaptive algorithms for the pathwise approximation of $M$
is rather different.
In \cite{Calvtcs:2007}, it was shown that for any (adaptive) algorithm
using $n$ function evaluations the average error is at least of order
$\exp(-c\, n/\log(n))$ for some positive constant~$c$.
In contrast to the nonadaptive case, we are unaware of algorithms
with error bounds matching the lower bound for adaptive algorithms.
In this paper we analyze an adaptive algorithm that has an average error
at most of order $n^{-r}$, for any positive number~$r$. Hence this algorithm
converges at an arbitrarily high polynomial rate.
In \cite{Calvin2011}, the same algorithm was shown to converge
in a probabilistic sense.
We are unaware of previous results showing the increased power of
adaptive methods relative to nonadaptive methods with respect to the $L_p$ error.

\bigskip
Several optimization algorithms have been proposed that
use the Brownian motion as a model for an unknown function to be
minimized, including \cite{kushner,Mockus:1972,Zi:1985,Calvin1997}.
One of the ideas proposed
in \cite{kushner} is to evaluate the function next at the
point where the function has the maximum conditional
probability of having a value less than the minimum of the conditional mean,
minus some positive amount (tending to zero).
This is the same idea behind our algorithm, described in Section 2.
The question of convergence of such (Bayesian) methods in general
is addressed in \cite{Mockus:1974}.
Several algorithms, with an emphasis on the question of convergence,
are described in \cite{ToZi:1989}.

In global optimization, the function to be optimized is typically assumed
to be a member of some class of functions. Often, the worst-case error
of algorithms on such a class of functions is studied.
However, if the function class is convex and symmetric, then the worst-case
error for any method using $n$ function evaluations is at
least as large as the error of a suitable nonadaptive method
using $n+1$ evaluations, see, e.g., \cite[Chap.~1.3]{novak}. In this case, a worst-case
analysis cannot justify the use of adaptive algorithms for global optimization.
An average-case analysis, where it is assumed that the function to be optimized
is drawn from a probability distribution, is an alternative to justify
adaptive algorithms for general function classes. Brownian motion is suitable
for such an average-case study since its analysis is tractable, yet the answers to the
complexity questions are far from obvious. As already explained, adaptive methods
are much more powerful than nonadaptive methods for
optimization of Brownian motion.

\bigskip
This paper is organized as follows. In Section~\ref{sec:algo} we present
our algorithm with corresponding error bound, see~Theorem~\ref{thm:main}.
In Section~\ref{sec:numeric} we illustrate our results by numerical experiments.
The rest of the paper is devoted to proving Theorem~\ref{thm:main}.

%-------------------------------------------------------------------
%-------------------------------------------------------------------
% Algorithm and Main Result
%-------------------------------------------------------------------
%-------------------------------------------------------------------
\section{Algorithm and Main Result}\label{sec:algo}
Let $f\colon [0,1]\to\R$ be a continuous function with $f(0)=0$.
We will recursively define a sequence
\begin{align}\label{eq:sequence}
	t_0,t_1,\ldots\in[0,1]
\end{align}
of pairwise distinct points from the unit interval.
These points are chosen adaptively, i.e., the $k$-th evaluation site $t_k$
may depend on the previous values $t_0,f(t_0),\ldots,t_{k-1},f(t_{k-1})$.
We use the discrete minimum over these points given by
\begin{align*}
	M_n = \min_{0\leq i\leq n} f(t_i)
\end{align*}
for $n\in\N_0$, as an approximation of the global minimum
\begin{align*}
	M = \inf_{0\leq t\leq 1}f(t)
\end{align*}
of $f$.
The aim is that
$M_n$ is a ``good'' approximation of $M$ on average if $f$ is a Brownian motion.

%-------------------------------------------------------------------
%-------------------------------------------------------------------
\bigskip
We begin by introducing some notation.
For $n\in\N_0$ we denote the ordered first $n$ evaluation sites by
\begin{align*}
	0 \leq  t_0^n < t_1^n < \ldots < t_n^n \leq  1
\end{align*}
such that $\{t^n_i:0\leq i\leq n\}=\{t_i: 0\leq i \leq n\}$.
Furthermore, for $n\in\N$ let
\begin{align*}
	\tau_n = \min_{1\leq i\leq n} t^n_i-t^n_{i-1}
\end{align*}
be the smallest distance between two evaluation sites.
Moreover, we define $g\colon{]0,1{]}}\to{{[}0,\infty[}$ by
\begin{align*}
	g(x) = \sqrt{\lambda x \log(1/x)} \,,
\end{align*}
where $\log$ denotes the natural logarithm.
Here, $\lambda\in{[1,\infty[}$ is a fixed parameter, which is convenient to
be left unspecified at this point.

Now, we define the sequence appearing in \eqref{eq:sequence}.
The first two evaluation sites are nonadaptively chosen to be $t_1=1$
and $t_2=1/2$. Moreover, for notational convenience we set $t_0=0$.

Let $n\geq2$, and suppose that the algorithm has already constructed
the first $n$ points $t_0,\dots,t_n$. The key quantity for choosing
the next evaluation site is given by
\begin{align}\label{eq:rho-i}
	\rho_i^n = \frac{ t^n_i-t^n_{i-1} }{
			\left( f(t^n_{i-1})-M_n+g(\tau_n) \right)
			\left( f(t^n_{i})-M_n+g(\tau_n) \right)
		}
\end{align}
for $i\in\{1,\ldots,n\}$. The algorithm splits the interval with the
largest value of $\rho^n_i$ at the midpoint.
More precisely, let $j\in\{1,\ldots,n\}$ be the smallest index
such that $\rho^n_{j}=\rho^n$ where
\begin{align*}
	\rho^n = \max_{1\leq i\leq n} \rho^n_i.
\end{align*}
The next function evaluation is then made at the midpoint
\begin{align*}
	t_{n+1} = \frac 12\left( t^n_{{j-1}}+t^n_{{j}} \right)
\end{align*}
of the corresponding subinterval.

As we use the discrete minimum $M_n$ as an approximation of the global
minimum $M$, the error of the proposed algorithm is given by
\begin{align*}
	\Delta_{n} = \Delta_{n,\lambda}(f) = M_n-M
\end{align*}
for $n\in\N_0$.

We stress that all quantities defined above depend on the prespecified
choice of
the parameter $\lambda\in{[1,\infty[}$. In particular, $\lambda$ affects
all adaptively chosen evaluation sites $t_3,t_4,\ldots$ and hence $M_n$.
However, we often do not explicitly indicate this dependence. 

The following theorem shows that this algorithm achieves an arbitrarily
high polynomial convergence rate w.r.t.~the $L_p$-norm in case of a
Brownian motion $W={(W(t))_{0\leq t\leq 1}}$.

\begin{thm}\label{thm:main}
	For all $r\in{[1,\infty[}$ and for all $p\in{[1,\infty[}$ there exist
	$\lambda\in{[1,\infty[}$ and $c>0$ such that
	\begin{align*}
		\left( \E{ \abs{ \Delta_{n,\lambda}(W) }^p } \right)^{1/p} \leq c\cdot n^{-r}
	\end{align*}
	for all $n\in\N$.
\end{thm}

\begin{rem}\label{rem:choice-lambda}
	Our analysis shows that
	\begin{align*}
		\lambda \geq 144\cdot (1+p\cdot r)
	\end{align*}
	is sufficient to obtain convergence order $r$ w.r.t.~the $L_p$-norm
	in Theorem~\ref{thm:main}. However, numerical experiments indicate
	an exponential decay even for small values of $\lambda$,
	see Figure~\ref{fig:p2}.
\end{rem}

\begin{rem}
	The number of function evaluations made by the algorithm to produce
	the approximation $M_n$ is a fixed number $n\in\N$
	(we assume that $f(0)=0$ and so we do not count $t_0$).
	Thus we do not consider adaptive stopping rules.
	A straightforward implementation of this algorithm on a computer
	requires operations of order $n^2$.
\end{rem}

An intuitive explanation why this algorithm works in the case of Brownian
motion is as follows. The function $g$ is chosen such that $M_n-M\le g(\tau_n)$
with high probability if $f$ is a Brownian path.
The idea of the algorithm is to next evaluate the function
at the midpoint of the subinterval that is most likely to
have a value less than $M_n-g(\tau_n)$.
Conditional on the values observed up to time $n$, the
probability that the minimum over
$[t^n_{i-1},t^n_i]$ is less than $M_n-g(\tau_n)$ is
\begin{align}\label{eq:underProb}
	\exp\left(-\frac{2}{\rho^n_i}\right),
\end{align}
see \cite{Calvin2001}.
The behavior of the $\{\rho^n_i\}$ defined in~\eqref{eq:rho-i} is
more convenient to characterize under the proposed algorithm than
the probabilities given in~\eqref{eq:underProb}.

The proof of Theorem~\ref{thm:main} relies on two sets of preliminary results.
Section~\ref{sec:deterministic} establishes upper bounds for the error when the algorithm
is applied to certain sets of functions, culminating in Corollary~\ref{cor:mainproperror}.
In Section~\ref{sec:probabilistic}, we bound the Wiener measure of these sets of functions,
leading to Corollary~\ref{cor:mainlowerboundprop}.
Section~\ref{sec:mainProof}
combines these results to prove Theorem~\ref{thm:main}.

%-------------------------------------------------------------------
%-------------------------------------------------------------------
% Numerical Results
%-------------------------------------------------------------------
%-------------------------------------------------------------------
\section{Numerical Results}\label{sec:numeric}
In this section we present numerical results of the proposed algorithm
for different values of the parameter $\lambda$.
Figure~\ref{fig:nle} shows the error $\Delta_n$ for each of three
independently generated Wiener paths using $\lambda=1$.

\begin{figure}[htp]
	\centering
	\includegraphics[width=0.9\textwidth]{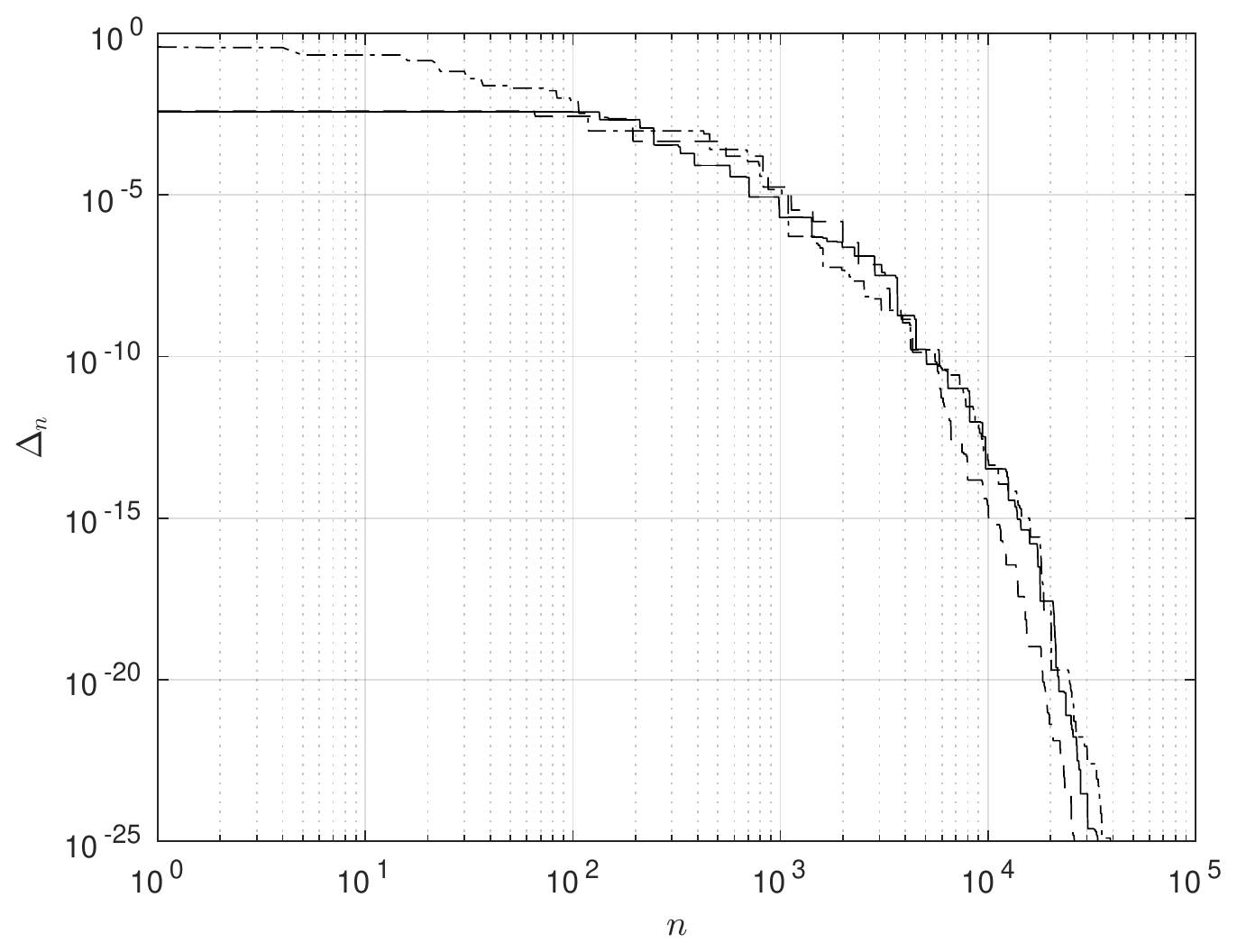}
	\caption{Errors for $3$ sample paths using $\lambda=1$.}
	\label{fig:nle}
\end{figure}

We also performed numerical experiments to estimate
$\left( \E{ \abs{ \Delta_{n,\lambda}(W) }^p } \right)^{1/p}$
using $10^3$ replications.
Figure~\ref{fig:p2} shows the results for $p=2$ and $\lambda\in\{1,4,8\}$.
We observe an exponential decay of the $L_2$ error for each value of $\lambda$.
Let us recall that Theorem~\ref{thm:main} and Remark~\ref{rem:choice-lambda}
only show that sufficiently large values of $\lambda$ ensure a ``high''
polynomial convergence rate of the $L_p$ error. However,
from a numerical point of view one might prefer choosing a small $\lambda$
since the numerically observed error in Figure~\ref{fig:p2} is increasing in~$\lambda$
for a fixed number of evaluation sites.
Let us mention that a small $\lambda$ corresponds to a small
offset $g(\tau_n)$ to the discrete minimum $M_n$ in~\eqref{eq:rho-i}.
Hence a small $\lambda$ results in a ``more local search''
around the discrete minimum.

\begin{figure}[htp]
	\centering
	\includegraphics[width=0.9\textwidth]{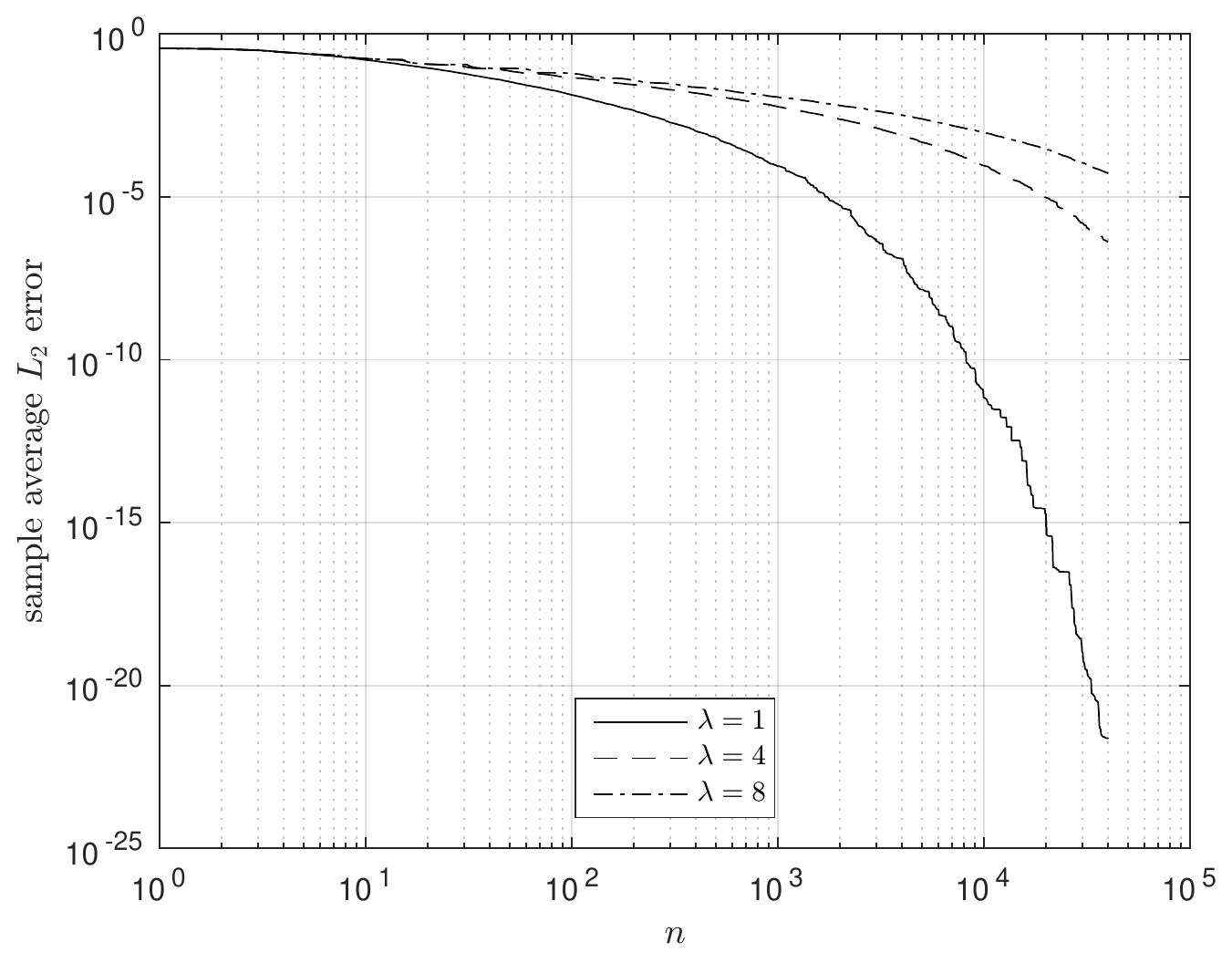}
	\caption{Sample $L_2$ error for various $\lambda$.}
	\label{fig:p2}
\end{figure}

%-------------------------------------------------------------------
%-------------------------------------------------------------------
% Non-probabilistic Arguments
%-------------------------------------------------------------------
%-------------------------------------------------------------------
\section{Non-probabilistic Arguments}\label{sec:deterministic}
In this section we will define a sequence of subsets of ``favorable"
functions for which we show that the error of the algorithm decreases
at an exponential rate.

%-------------------------------------------------------------------
%-------------------------------------------------------------------
\bigskip
First, let us mention some basic facts,
which will be frequently used in this paper.
Due to the bisection strategy, the lengths of all subintervals satisfy
\begin{align*}
	t^n_i-t^n_{i-1} \in\mathcal{A} = \{1/2^{k}:k\in\N\}
\end{align*}
for all $n\geq 2$ and $i\in\{1,\ldots,n\}$, and consequently
$\tau_n\in\mathcal{A}$ for all $n\geq 2$.
Let us stress that $g$ is non-decreasing on $\mathcal{A}$.
Furthermore, we have $\lim_{x\to0}g(x)=0$.

%-------------------------------------------------------------------
%-------------------------------------------------------------------
\bigskip
Let
\begin{align*}
	F = \{ f\colon [0,1]\to\R,\ f\text{ continuous with } f(0)=0 \}.
\end{align*}
Moreover, for $n\geq2$ and $\lambda\in{[1,\infty[}$ we define
\begin{align*}
	F_n=F_{\lambda,n} =\left\{ f\in F: \max_{1\leq k\leq n}\max_{1\leq i\leq k}
			\frac{ \abs{f(t^k_i)-f(t^k_{i-1})} }{ \sqrt{t^k_i-t^k_{i-1}} }
		\leq \sqrt{\lambda\log(n)/4} \right\}.
\end{align*}
The sets of ``favorable'' functions, defined in~\eqref{eq:favorable} below,
will be the intersection of several sets, including $F_n$, which depend
on the prespecified parameter $\lambda$ of the algorithm.
To simplify the notation, we will suppress the dependence of these
sets on $\lambda$ after their definition.
Recall that most quantities defined above depend on $\lambda$, $n$, and $f$
simultaneously. However, we typically only highlight the dependence on $n$.
For instance, $\rho_i^n$ also depends on the corresponding function $f$ as well
as on the parameter $\lambda$.

In the following we present some properties of the algorithm applied to
functions $f\in F_{n}$, which will be frequently used in this paper.

\begin{lem}\label{lem:rho}
	For all $\lambda\geq 1$, $n\geq 2$, and $f\in F_{n}$ we have
	\begin{align*}
		\rho^n \leq \frac{2}{\lambda \log(1/\tau_n)}.
	\end{align*}
	In particular, $\rho^n\leq 2/(\lambda\log(n))$.
\end{lem}
\begin{proof}
	First, we observe that
	\begin{align}\label{eq:split-smallest}
		\rho^m \leq \frac{\tau_m}{g(\tau_m)^2}
			= \frac{1}{\lambda\log(1/\tau_m)},
	\end{align}
	whenever the algorithm is about to split a smallest subinterval at step $m\geq2$.
	In the following step $m+1$, we also clearly have
	\begin{align*}
		\rho^{m+1}_i &\leq \frac{\tau_{m+1}}{g(\tau_{m+1})^2}
			= \frac{1}{\lambda\log(1/\tau_{m+1})},
	\end{align*}
	if $i\in\{1,\ldots,m+1\}$ corresponds to one of the newly created
	smallest subintervals. If $j\in\{1,\ldots,m+1\}$ denotes
	a subinterval that has not been split at step $m$ we obtain
	\begin{align*}
		\rho^{m+1}_j &= \frac{ t^{m+1}_j-t^{m+1}_{j-1} }
			{ \left( f(t^{m+1}_{j-1})-M_{m+1}+g(\tau_{m+1}) \right)
				\left( f(t^{m+1}_{j})-M_{m+1}+g(\tau_{m+1}) \right)
			} \\
		&\leq \left( \frac{ g(\tau_{m}) }{ g(\tau_{m+1}) } \right)^2
			\cdot \frac{ t^{m+1}_j-t^{m+1}_{j-1} }
			{ \left( f(t^{m+1}_{j-1})-M_{m}+g(\tau_{m}) \right)
				\left( f(t^{m+1}_{j})-M_{m}+g(\tau_{m}) \right)
			} \\
		&\leq \left( \frac{ g(\tau_{m}) }{ g(\tau_{m+1}) } \right)^2
			\cdot \rho^m,
	\end{align*}
	since $g$ is non-decreasing on $\mathcal{A}$.
	Moreover, $\tau_{m+1}=\tau_m/2$ and \eqref{eq:split-smallest}
	imply
	\begin{align*}
		\rho^{m+1}_j \leq \frac{ 2\tau_{m+1}\log(1/\tau_{m}) }{ \tau_{m+1}\log(1/\tau_{m+1}) }
			\cdot \frac{1}{ \lambda\log(1/\tau_m) }
			= \frac{2}{ \lambda\log(1/\tau_{m+1}) }
	\end{align*}
	and thus
	\begin{align}\label{eq:after-split-smallest}
		\rho^{m+1} \leq \frac{2}{ \lambda\log(1/\tau_{m+1}) }.
	\end{align}

	Let $n\geq2$ be arbitrary and let $m\in\{2,\ldots,n\}$ be the last time that the algorithm
	was about to split a smallest subinterval, thus \eqref{eq:split-smallest} holds.
	Let us stress that $\tau_{m+k}=\tau_n$ for all $k\in\{1,\ldots,n-m\}$.
	We will show by induction that
	\begin{align}\label{eq:induction}
		\rho^{m+k} \leq \frac{2}{\lambda \log(1/\tau_{m+k})}
	\end{align}
	for all $k\in\{0,\ldots,n-m\}$.
	
	We consider the non-trivial case of $m<n-1$, and we assume that \eqref{eq:induction} holds for some
	${k\in\{1,\ldots,n-m-1\}}$. At iteration $m+k+1$, we suppose that the $i$-th subinterval was split
	at step $m+k$, thus $\rho^{m+k}=\rho_i^{m+k}$, and consequently $t_{m+k+1}=(t^{m+k}_{i-1}+t^{m+k}_i)/2$.
	Then we have
	\begin{align*}
		f(t_{m+k+1}) = \frac{f(t^{m+k}_{i-1})+f(t^{m+k}_i)}{2} + \delta ,
	\end{align*}
	for some $\delta\in\R$.
	Furthermore, $f\in F_{n}$ implies
	\begin{align*}
		\abs{\frac{f(t^{m+k}_{i-1})+f(t^{m+k}_i)}{2} + \delta - f(t^{m+k}_{i-1})}
			\leq \sqrt{T_i/2}\cdot \sqrt{\lambda\log(n)/4}
	\end{align*}
	and
	\begin{align*}
		\abs{\frac{f(t^{m+k}_{i-1})+f(t^{m+k}_i)}{2} + \delta - f(t^{m+k}_{i})}
			\leq \sqrt{T_i/2}\cdot \sqrt{\lambda\log(n)/4}
	\end{align*}
	for $T_i=t^{m+k}_i-t^{m+k}_{i-1}$. In particular, this yields
	\begin{align}\label{eq:delta-bound}
		\delta \geq \frac{ \abs{f(t^{m+k}_{i})-f(t^{m+k}_{i-1})} }{2}
			- \sqrt{T_i/2}\cdot \sqrt{\lambda\log(1/\tau_n)/4}
	\end{align}
	since $\tau_n\leq 1/n$. We obtain
	\begin{align*}
		\frac{ \rho^{m+k}_i }{ \rho^{m+k+1}_i }
			&= 2\cdot \frac{
				\left(f(t^{m+k}_{i-1})-M_{m+k+1}+g(\tau_n)\right)
				\left(f(t_{m+k+1})-M_{m+k+1}+g(\tau_n)\right)
				}{
				\left(f(t^{m+k}_{i-1})-M_{m+k}+g(\tau_n)\right)
				\left(f(t^{m+k}_{i})-M_{m+k}+g(\tau_n)\right)
				}\\
		&\geq \frac{
			2\left(f(t_{m+k+1})-M_{m+k+1}+g(\tau_n)\right)
			}{
			\left( f(t^{m+k}_{i})-M_{m+k}+g(\tau_n) \right)
			}.
	\end{align*}
	Moreover,  we have
	\begin{align*}
		2 & \left( f(t_{m+k+1})-M_{m+k+1}+g(\tau_n) \right) \\
		&\geq f(t^{m+k}_{i-1})+f(t^{m+k}_i) + \abs{f(t^{m+k}_{i})-f(t^{m+k}_{i-1})}
			- \sqrt{\lambda T_i\log(1/\tau_n)/2} - 2M_{m+k} + 2g(\tau_n) \\
		&= \left( f(t^{m+k}_{i-1})-M_{m+k}+g(\tau_n) \right)
			+ \left( f(t^{m+k}_i)-M_{m+k}+g(\tau_n) \right) \\
		&\quad+ \abs{ \left( f(t^{m+k}_i)-M_{m+k}+g(\tau_n) \right) - \left( f(t^{m+k}_{i-1})-M_{m+k}+g(\tau_n) \right) }
			- \sqrt{\lambda T_i\log(1/\tau_n)/2} 
	\end{align*}
	due to \eqref{eq:delta-bound}. Making the substitution
	\begin{align*}
		a = f(t^{m+k}_{i})-M_{m+k}+g(\tau_n) > 0
			\quad\wedge\quad
			b = f(t^{m+k}_{i-1})-M_{m+k}+g(\tau_n) > 0
	\end{align*}
	and $x=\sqrt{a/b}>0$ we get
	\begin{align*}
		\frac{ \rho^{m+k}_i }{ \rho^{m+k+1}_i } &\geq \frac{
				a + b +\abs{a-b}
				- \sqrt{T_i}\cdot \sqrt{\lambda\log(1/\tau_n)/2}
			}{a}\\
		&= \frac{
				x + \frac{1}{x} +\abs{x-\frac{1}{x}}
				- \sqrt{\rho_i^{m+k}}\cdot \sqrt{\lambda\log(1/\tau_n)/2}
			}{x}\\
		&\geq \frac{
				x + \frac{1}{x} + \abs{x-\frac{1}{x}}
				-1
			}{x}\\
		&\geq 1,
	\end{align*}
	where the second inequality holds by the induction hypothesis.
	Hence we get
	\begin{align*}
		{ \rho^{m+k+1}_i }
			\leq { \rho^{m+k}_i }.
	\end{align*}
	Similarly, we obtain
	$\rho^{m+k+1}_{i+1}\leq\rho^{m+k}_i$ and thus $\rho^{m+k+1}\leq\rho^{m+k}$.
\end{proof}

\begin{lem}\label{lem:AA}
	For all $\lambda\geq 1$, $n\geq 2$, and $f\in F_{n}$ we have
	\begin{align*}
		\max_{1\leq i\leq n} \frac{ t^n_i-t^n_{i-1} }{ (\min\{f(t^n_{i-1}),f(t^n_i)\}-M_n+g(\tau_n))^2 }
			\leq \frac{4}{\lambda\log(n)}.
	\end{align*}
\end{lem}
\begin{proof}
	Let $n\geq2$. At first, observe that
	\begin{align*}
		\frac{t^n_i-t^n_{i-1}}{ (\min\{f(t^n_{i-1}),f(t^n_i)\}-M_n+g(\tau_n))^2 }
			&= \rho_i^n \cdot \left(
					1+\frac{ \abs{f(t^n_i)-f(t^n_{i-1})} }{ \min\{f(t^n_{i-1}),f(t^n_i)\}-M_n+g(\tau_n) }
				\right)
	\end{align*}
	for $i\in\{1,\dots,n\}$, and thus
	\begin{align*}
		&\max_{1\leq i\leq n} \frac{ t^n_i-t^n_{i-1} }{ (\min\{f(t^n_{i-1}),f(t^n_i)\}-M_n+g(\tau_n))^2 } \\
		&\qquad\leq \rho^n \cdot \left(
				1+\max_{1\leq i\leq n} \frac{ \abs{f(t^n_i)-f(t^n_{i-1})} }{ \min\{f(t^n_{i-1}),f(t^n_i)\}-M_n+g(\tau_n) }
			\right) \\
		&\qquad= \rho^n \cdot \left(
				1+\max_{1\leq i\leq n} \frac{ \abs{f(t^n_i)-f(t^n_{i-1})}\cdot\sqrt{t^n_i-t^n_{i-1}} }
				{ \sqrt{t^n_i-t^n_{i-1}}\cdot \left(\min\{f(t^n_{i-1}),f(t^n_i)\}-M_n+g(\tau_n)\right) }
			\right) \\
		&\qquad\leq \rho^n \cdot \left(
				1+\max_{1\leq i\leq n} \frac{ \abs{f(t^n_i)-f(t^n_{i-1})} }{ \sqrt{t^n_i-t^n_{i-1}} }
				\cdot \max_{1\leq i\leq n} \frac{ \sqrt{t^n_i-t^n_{i-1}} }{ \min\{f(t^n_{i-1}),f(t^n_i)\}-M_n+g(\tau_n) }
			\right).
	\end{align*}
	Setting
	\begin{align*}
		z_n = \max_{1\leq i\leq n} \frac{ \sqrt{t^n_i-t^n_{i-1}} }{ \min\{f(t^n_{i-1}),f(t^n_i)\}-M_n+g(\tau_n) },
	\end{align*}
	the last inequality reads
	\begin{align*}
		z_n^2 \leq \rho^n \cdot \left(
				1 + z_n\cdot \max_{1\leq i\leq n} \frac{ \abs{f(t^n_i)-f(t^n_{i-1})} }{ \sqrt{t^n_i-t^n_{i-1}} }
			\right).
	\end{align*}
	Now use the fact that $z_n>0$, $f\in F_{n}$ and $\rho^n\leq 2/(\lambda\log(n))$
	from Lemma~\ref{lem:rho} to obtain
	\begin{align*}
		z_n^2 \leq \frac{2}{\lambda\log(n)}\cdot \left( 1 + z_n\cdot \sqrt{\lambda\log(n)/4} \right),
	\end{align*}
	or equivalently,
	\begin{align*}
		\left( z_n - 1/\sqrt{4\lambda\log(n)} \right)^2
			\leq \left(2+\frac 14\right) \cdot \frac{1}{\lambda\log(n)}.
	\end{align*}
	This implies that
	\begin{align*}
		z_n \leq \frac{1}{\sqrt{\lambda\log(n)}} \left( 1/2 + \sqrt{2+1/4} \right)
			= \frac{2}{\sqrt{\lambda\log(n)}}
	\end{align*}
	and hence $z_n^2\leq 4/(\lambda\log(n))$.
\end{proof}

\begin{lem}\label{lem:lower-bound}
	For all $\lambda\geq 1$, $n\geq 2$, and $f\in F_{n}$ we have
	\begin{align*}
		\rho^k \geq \frac{2}{ 3\lambda\log(1/\tau_n) }
	\end{align*}
	for all $k\in\{2,\dots,n\}$.
\end{lem}
\begin{proof}
	Let $n\geq2$, $k\in\{2,\dots,n\}$ and $f\in F_{n}$. Moreover, let $i\in\{1,\dots,k\}$ be an index with
	$M_k\in\{f(t^k_{i-1}),f(t^k_i)\}$ and $T_i=t^k_i-t^k_{i-1}$. Then we have
	\begin{align*}
		\rho^k \geq \rho^k_i
			&= \frac{ T_i }{ g(\tau_k)\cdot\left( \max\{f(t^k_{i-1}),f(t^k_i)\}-\min\{f(t^k_{i-1}),f(t^k_i)\}+g(\tau_k) \right) } \\
		&= \frac{1}{ \sqrt{ \lambda\log(1/\tau_k)\cdot(\tau_k/T_i) }
				\cdot\left( \abs{f(t^k_i)-f(t^k_{i-1})}/\sqrt{T_i} + \sqrt{ \lambda\log(1/\tau_k)\cdot(\tau_k/T_i) } \right)
			} \\
		&\geq \frac{1}{ \sqrt{\lambda \log(1/\tau_k)}
			\left( \sqrt{\lambda\log(n)/4} + \sqrt{\lambda
			\log(1/\tau_k)} \right) }\\
	&\geq \frac{1}{ \sqrt{\lambda \log(1/\tau_n)}
			\left( \sqrt{\lambda\log(1/\tau_n)/4} + \sqrt{\lambda \log(1/\tau_n)} \right) } \\
		&\geq \frac{2}{ 3\lambda\log(1/\tau_n)
		},
	\end{align*}
	since $\tau_n\leq 1/n$ and $\tau_n\leq\tau_k$.
\end{proof}

%-------------------------------------------------------------------
% Upper Bound
%-------------------------------------------------------------------
\subsection{Upper Bound on \texorpdfstring{$\sum_{i=1}^n \rho^n_i$}{Rho}}
For $\lambda\in{[1,\infty[}$, $n\in\N$, and $f\in F$ we
consider the linear interpolation $L_n$ of the $\{f(t^n_i)\}$ defined by
\begin{align}\label{eq:lin-int}
	L_n(s) = \frac{t^n_i-s}{t^n_i-t^n_{i-1}} \cdot f(t^n_{i-1})
		+ \frac{s-t^n_{i-1}}{t^n_i-t^n_{i-1}} \cdot f(t^n_{i}),
		\quad s\in [t^n_{i-1}, t^n_i],\quad 1\leq i\leq n.
\end{align}

\begin{rem}
	A simple computation shows that
	\begin{align*}
		\int_s^t \frac{1}{(h(x))^2}\dx
			=\frac{t-s}{h(s)\cdot h(t)}
	\end{align*}
	for $s<t$ and $h\colon [s,t]\to\R$ affine linear with $h(s),h(t)>0$.
	This yields
	\begin{align*}
		\rho^n_i = \int_{t^n_{i-1}}^{t^n_i} \frac{1}{(L_n(s)-M_n+g(\tau_n))^2} \ds
	\end{align*}
	and hence
	\begin{align}\label{eq:rho-sum}
		\sum_{i=1}^n \rho^n_i
			= \int_{0}^1 \frac{1}{(L_n(t)-M_n+g(\tau_n))^2} \dt
	\end{align}
	for $f\in F$, $n\geq 2$, and $\lambda\geq 1$.
\end{rem}

Replacing the discrete minimum by the global minimum in \eqref{eq:rho-sum} clearly yields the lower bound
\begin{align*}
	\int_{0}^1 \frac{1}{(L_n(t)-M+g(\tau_n))^2} \dt
		\leq \sum_{i=1}^n \rho^n_i.
\end{align*}
In the following we provide an upper bound of similar structure.

For $n\geq2$ and $\lambda\in{[1,\infty[}$ we define
\begin{align*}
	G_{1,n} = G_{\lambda,1,n}
		= \{ f\in F:\,\Delta_n\leq g(\tau_n) \}
\end{align*}
and
\begin{align*}
	G_{{1}/{2},n} = G_{\lambda,1/2,n}
		= \left\{ f\in F:\,\Delta_n\leq \frac12\,g(\tau_n) \right\}.
\end{align*}
We clearly have $G_{{1}/{2},n}\subseteq G_{1,n}$.

\begin{lem}\label{lem:upper-bound-1det}
	For all $\lambda\geq 1$, $n\geq 2$, and $f\in G_{{1}/{2},n}$ we have
	\begin{align*}
			\sum_{i=1}^n \rho^n_i \leq 4\cdot\int_{0}^1 \frac{1}{(L_n(t)-M+g(\tau_n))^2} \dt.
	\end{align*}
\end{lem}
\begin{proof}
	Using \eqref{eq:rho-sum} we obtain
	\begin{align*}
		\sum_{i=1}^n \rho^n_i
			&= \int_{0}^1 \frac{1}{(L_n(t)-M+g(\tau_n)-\Delta_n)^2} \dt \\
		&\leq \int_{0}^1 \frac{1}{\left( L_n(t) - M +
		\frac{1}{2}\,g(\tau_n) \right)^2} \dt \\
		&\leq \frac{1}{\left(1/2\right)^2} \cdot
			\int_{0}^1 \frac{1}{(L_n(t)-M+g(\tau_n))^2} \dt.
	\end{align*}
\end{proof}

In the next step we bound $\sum_{i=1}^n\rho^n_i$ in terms of
\begin{align*}
	\int_{0}^1 \frac{1}{(f(t)-M+g(\tau_n))^2} \dt.
\end{align*}

For $n\geq 2$ and $\lambda\in{[1,\infty[}$ we define
\begin{align}\label{eq:deff+}
	J_{n}^+ = 
	J_{\lambda,n}^+ = \left\{
	f\in F:\,
			\max_{1\leq i\leq n}\, \sup_{s\in[t^n_{i-1},t^n_i]}
			\frac{f(s)-L_n(s)}{\sqrt{t^n_i-t^n_{i-1}}} \leq \frac 12\sqrt{\lambda\log(n)/4}
		\right\}
\end{align}
and
\begin{align}\label{eq:deff-}
	J_{n}^- = 
	J_{\lambda,n}^- = \left\{
	f\in F:\,
			\min_{1\leq i\leq n}\, \inf_{s\in[t^n_{i-1},t^n_i]}
			\frac{f(s)-L_n(s)}{\sqrt{t^n_i-t^n_{i-1}}} \geq -\frac 12\sqrt{\lambda\log(n)/4}
		\right\}.
\end{align}

\begin{lem}\label{lem:upper-bound-2det}
	For all $\lambda\geq 1$, $n\geq 2$, and $f\in F_{n}\cap J_{n}^+\cap J_{n}^-$ we have
	\begin{align*}
		\int_{0}^1 \frac{1}{(L_n(t)-M+g(\tau_n))^2} \dt
			\leq \frac94\cdot\int_{0}^1 \frac{1}{(f(t)-M+g(\tau_n))^2} \dt.
	\end{align*}
\end{lem}
\begin{proof}
	Let $i\in\{1,\dots,n\}$. We have for all $s\in[t_{i-1}^n,t_i^n]$ that
	\begin{align*}
		\frac{ \abs{f(s)-L_n(s)} }{ L_n(s)-M+g(\tau_n) }
			&\leq \frac{ \abs{f(s)-L_n(s)} }{ \sqrt{t^n_i-t^n_{i-1}} }
				\cdot \frac{ \sqrt{t^n_i-t^n_{i-1}} }{ \min\{f(t^n_{i-1}),f(t^n_i)\}-M_n+g(\tau_n) } \\
		&\leq \frac 12\sqrt{\lambda\log(n)/4} \cdot \frac{2}{\sqrt{\lambda\log(n)}} \\
		&= \frac12
	\end{align*}
	due to Lemma~\ref{lem:AA}. This yields
	\begin{align*}
		\int_{0}^1 \frac{1}{ (f(t)-M+g(\tau_n))^2 } \dt
			&= \int_{0}^1 \frac{1}{ (L_n(t)-M+g(\tau_n))^2 \cdot \left(
					1 + \frac{ f(t)-L_n(t) }{ L_n(t)-M+g(\tau_n) }
				\right)^2
			} \dt \\
		&\geq \frac{4}{9} \cdot \int_{0}^1 \frac{1}{(L_n(t)-M+g(\tau_n))^2} \dt.
	\end{align*}
\end{proof}

For $\lambda\in{[1,\infty[}$, $n\geq 2$, and $C>0$ we define
\begin{align}\label{eq:defhcn}
	H_{C,n} = H_{\lambda,C,n}
		= \left\{ f\in F:\,\int_0^1\frac{1}{(f(t)-M+g(\tau_n))^2}\dt
			\leq C\cdot \left(\log\left({1}/{g(\tau_n)}\right)\right)^4
		\right\}.
\end{align}
Note that $H_{C_1,n}\subseteq H_{C_2,n}$, if $0<C_1\leq C_2$. 

\begin{cor}\label{cor:lowerbounddet}
	For all $\lambda\geq 1$, $n\geq 2$, $C>0$, and
	$f\in F_{n}\cap J_{n}^+\cap J_{n}^-\cap G_{{1}/{2},n}\cap H_{C,n}$ we have
	\begin{align*}
		\sum_{i=1}^n \rho^n_i
			\leq 9\cdot C\cdot \left(\log\left({1}/{g(\tau_n)}\right)\right)^4.
	\end{align*}
\end{cor}
\begin{proof}
	This follows directly from Lemma~\ref{lem:upper-bound-1det},
	Lemma~\ref{lem:upper-bound-2det}, and \eqref{eq:defhcn}.
\end{proof}

%-------------------------------------------------------------------
% Lower Bound
%-------------------------------------------------------------------
\subsection{Lower Bound on \texorpdfstring{$\sum_{i=1}^n \rho^n_i$}{Rho}}
By $\lceil \cdot \rceil$ we denote the ceiling function, e.g., $\lceil 2\rceil=2$ and $\lceil 5/2\rceil=3$.

\begin{prop}\label{prop:upperbounddet}
	For all $\lambda\geq 1$, $n\geq 4$, and
	$f\in\bigcap_{k=\lceil n/2\rceil}^n \left(G_{1,k}\cap F_{k}\right)$
	we have
	\begin{align*}
		\sum_{i=1}^n \rho^n_i \geq \frac{1}{288}\cdot\frac{n}{ \lambda\log(1/\tau_n) }.
	\end{align*}
\end{prop}
\begin{proof}
	At iteration $n$, there are $n$ subintervals
	\begin{align*}
	[t^n_0,t^n_1], [t^n_1,t^n_2],\ldots ,[t^n_{n-1},t^n_n].
	\end{align*}
	At first, we observe that at least $\lceil n/2\rceil$ of these subintervals resulted from the iterations
	$\lceil n/2\rceil,\ldots,n-1$. Now suppose that such a subinterval, say $\left[t^n_{i-1},t^n_i\right]$ for $i\in\{1,\dots,n\}$,
	resulted from the split of the interval $\left[t^{k_i}_{j_i-1},t^{k_i}_{j_i}\right]$ at the $k_i$-th iteration
	of the algorithm, i.e.,
	\begin{align*}
		\rho^{k_i}=\rho^{k_i}_{j_i} \quad\wedge\quad
			t_{k_i+1} = t^{k_i+1}_{j_i} = \left( t^{k_i}_{j_i-1}+t^{k_i}_{j_i} \right)/2
	\end{align*}
	with $k_i\in\{\lceil n/2\rceil,\ldots,n-1\}$ and $j_i\in\{1,\ldots,k_i\}$. Thus we have
	\begin{align*}
		\left[t^n_{i-1},t^n_i\right] = \left[ t^{k_i}_{j_i-1},t_{k_i+1} \right] = \left[ t^{k_i+1}_{j_i-1},t^{k_i+1}_{j_i} \right]
			\quad\vee\quad
			\left[t^n_{i-1},t^n_i\right] = \left[ t_{k_i+1},t^{k_i}_{j_i} \right] = \left[ t^{k_i+1}_{j_i},t^{k_i+1}_{j_i+1} \right],
	\end{align*}
	as depicted in Figure~\ref{fig:pic}.
	
\begin{figure}[htp]
	\centering
	\includegraphics[width=0.9\linewidth]{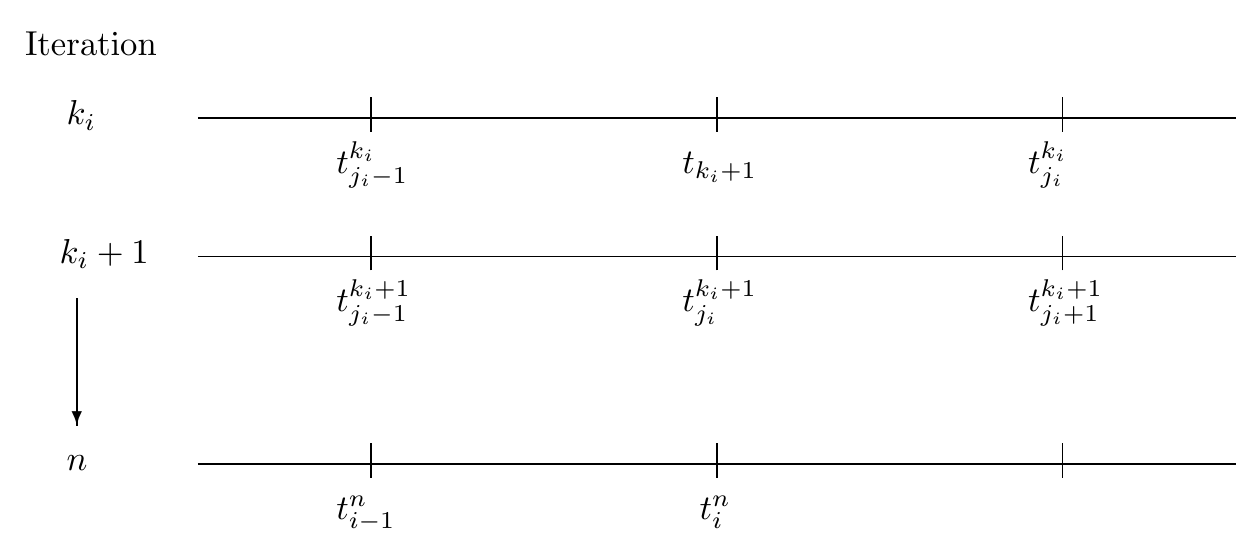}
	\caption{Situation from the proof of Proposition~\ref{prop:upperbounddet}.}
	\label{fig:pic}
\end{figure}

	Without loss of generality we may consider the case where
	$\left[t^n_{i-1},t^n_i\right]=\left[t^{k_i+1}_{j_i-1},t^{k_i+1}_{j_i}\right]$ is given by the left child.
	First we show
	\begin{align}\label{eq:firstinproof}
		\rho^{k_i+1}_{j_i} \geq \frac{1}{24}\,\rho^{k_i}.
	\end{align}
	For this we may without loss of generality assume
	that $f\left( t^{k_i}_{j_i-1} \right) \geq f\left( t^{k_i}_{j_i}
	\right)$ and hence $\rho^{k_i+1}_{j_i} \le \rho^{k_i+1}_{j_i+1}$.
	We have
	\begin{align*}
		\rho^{k_i+1}_{j_i} &= \frac{1}{2} \cdot
			\frac{ t^{k_i}_{j_i} - t^{k_i}_{{j_i}-1} }{
				\left( f(t^{k_i}_{{j_i}-1})-M_{k_i+1}+g(\tau_{k_i+1}) \right)
				\left( f(t_{k_i+1})-M_{k_i+1}+g(\tau_{k_i+1}) \right)
			} \\
		&= \frac{1}{2}\,\rho^{k_i} \cdot
			\frac{
				\left( f(t^{k_i}_{j_i-1})-M_{k_i}+g(\tau_{k_i}) \right)
				\left( f(t^{k_i}_{j_i})-M_{k_i}+g(\tau_{k_i}) \right)
			}{
				\left( f(t^{k_i}_{{j_i}-1})-M_{k_i+1}+g(\tau_{k_i+1}) \right)
				\left( f(t_{k_i+1})-M_{k_i+1}+g(\tau_{k_i+1}) \right)
			} \\
		&\geq \frac{1}{2}\,\rho^{k_i} \cdot
			\frac{
				\left( f(t^{k_i}_{j_i-1})-M_{k_i}+g(\tau_{k_i}) \right)
				\left( f(t^{k_i}_{j_i})-M_{k_i}+g(\tau_{k_i}) \right)
			}{
				\left( f(t^{k_i}_{{j_i}-1})-M_{k_i+1}+g(\tau_{k_i}) \right)
				\left( f(t_{k_i+1})-M_{k_i+1}+g(\tau_{k_i}) \right)
			},
	\end{align*}
	since $g$ is non-decreasing on $\mathcal{A}$.
	Moreover, $f\in F_{k_i}$ and $f(t^{k_i}_{j_i-1})\geq f(t^{k_i}_{j_i})$ imply
	\begin{align}\label{eq:prop3-1}
		1\leq \frac{ f(t^{k_i}_{j_i-1})-M_{k_i}+g(\tau_{k_i}) }
				{ f(t^{k_i}_{j_i})-M_{k_i}+g(\tau_{k_i}) }
			= 1 + \frac{ f(t^{k_i}_{j_i-1})-f(t^{k_i}_{j_i}) }
				{ \sqrt{ t^{k_i}_{j_i}-t^{k_i}_{j_i-1} } }
			\frac{ \sqrt{t^{k_i}_{j_i}-t^{k_i}_{j_i-1} } }
				{ f(t^{k_i}_{j_i})-M_{k_i}+g(\tau_{k_i}) }
			\leq 2
	\end{align}
	due to Lemma~\ref{lem:AA}. Analogously, $f\in F_{k_i+1}$ yields
	\begin{align*}
		f(t_{k_i+1}) \leq  f(t^{k_i}_{j_i-1})
			+ \sqrt{\lambda\log(k_i+1)/4} \cdot \sqrt{\left( t^{k_i}_{j_i}-t^{k_i}_{j_i-1} \right)/2}
	\end{align*}
	and thus ($n\geq4$ thus $k_i\geq 2$ and $\log(k_i+1)\leq2\log(k_i)$)
	\begin{align}\label{eq:prop3-2}
		\begin{aligned}[c]
			&\frac{ f(t_{k_i+1})-M_{k_i}+g(\tau_{k_i}) }{ f(t^{k_i}_{j_i})-M_{k_i}+g(\tau_{k_i}) } \\
				&\qquad\qquad\leq \frac{ f(t^{k_i}_{j_i-1})-M_{k_i}+g(\tau_{k_i}) }{ f(t^{k_i}_{j_i})-M_{k_i}+g(\tau_{k_i}) }
					+ \sqrt{\lambda\log(k_i)/4} \cdot
					\frac{ \sqrt{ t^{k_i}_{j_i}-t^{k_i}_{j_i-1} } }{ f(t^{k_i}_{j_i})-M_{k_i}+g(\tau_{k_i}) }
				\leq 3
		\end{aligned}
	\end{align}
	due to \eqref{eq:prop3-1} and Lemma~\ref{lem:AA}. Furthermore, $M_{k_i}-M\leq g(\tau_{k_i})$ shows
	\begin{align}\label{eq:prop3-3}
		0\leq \frac{ M_{k_i}-M_{k_i+1} }{ f(t^{k_i}_{j_i})-M_{k_i}+g(\tau_{k_i}) }
			\leq \frac{ M_{k_i}-M }{ g(\tau_{k_i}) }
			\leq 1.
	\end{align}
	Combining \eqref{eq:prop3-1}, \eqref{eq:prop3-2} and \eqref{eq:prop3-3} yields
	\begin{align*}
		&\frac{
				\left( f(t^{k_i}_{j_i-1})-M_{k_i}+g(\tau_{k_i}) \right)
				\left( f(t^{k_i}_{j_i})-M_{k_i}+g(\tau_{k_i}) \right)
			}{
				\left( f(t^{k_i}_{{j_i}-1})-M_{k_i+1}+g(\tau_{k_i}) \right)
				\left( f(t_{k_i+1})-M_{k_i+1}+g(\tau_{k_i}) \right)
			} \\
		&\qquad= \frac{
				\frac{ f(t^{k_i}_{j_i-1})-M_{k_i}+g(\tau_{k_i}) }
				{ f(t^{k_i}_{j_i})-M_{k_i}+g(\tau_{k_i}) }
			}{
				\left( \frac{ f(t^{k_i}_{{j_i}-1})-M_{k_i}+g(\tau_{k_i}) }{ f(t^{k_i}_{j_i})-M_{k_i}+g(\tau_{k_i}) }
					+ \frac{ M_{k_i}-M_{k_i+1} }{ f(t^{k_i}_{j_i})-M_{k_i}+g(\tau_{k_i}) }
				\right)
				\left( \frac{ f(t_{k_i+1})-M_{k_i}+g(\tau_{k_i}) }{ f(t^{k_i}_{j_i})-M_{k_i}+g(\tau_{k_i}) }
					+ \frac{ M_{k_i}-M_{k_i+1} }{ f(t^{k_i}_{j_i})-M_{k_i}+g(\tau_{k_i}) }
				\right)
			} \\
		&\qquad\geq \frac{1}{(2+1)\cdot(3+1)}
	\end{align*}
	and hence we get \eqref{eq:firstinproof}.
	
	Now, we exploit
	\begin{align*}
		\rho^n_i &= \rho^{k_i+1}_{j_i} \cdot
			\frac{ \left( f(t^{k_i}_{{j_i}-1})-M_{k_i+1}+g(\tau_{k_i+1}) \right)
				\left( f(t_{k_i+1})-M_{k_i+1}+g(\tau_{k_i+1}) \right)
			}{
				\left( f(t^{k_i}_{{j_i}-1})-M_{n}+g(\tau_{n}) \right)
				\left( f(t_{k_i+1})-M_{n}+g(\tau_{n}) \right)
			} \\
		&\geq \rho^{k_i+1}_{j_i} \cdot
			\frac{ \left( f(t^{k_i}_{{j_i}-1})-M_{k_i+1}+g(\tau_{k_i+1}) \right)
			}{
				\left( f(t^{k_i}_{{j_i}-1})-M_{k_i+1}+g(\tau_{k_i+1}) + \left(M_{k_i+1}-M_n\right) \right)
			} \\
		&\qquad\qquad\qquad\qquad\qquad\qquad\cdot
			\frac{ \left( f(t_{k_i+1})-M_{k_i+1}+g(\tau_{k_i+1}) \right)
			}{
				\left( f(t_{k_i+1})-M_{k_i+1}+g(\tau_{k_i+1}) + \left(M_{k_i+1}-M_n\right) \right)
			} \\
		&\geq \rho^{k_i+1}_{j_i} \cdot
			\frac{1}{1+1}\cdot\frac{1}{1+1},
	\end{align*}
	where the last inequality holds due to $M_{k_i+1}-M_n\leq M_{k_i+1}-M\leq g(\tau_{k_i+1})$, and use Lemma~\ref{lem:lower-bound} to conclude that
	\begin{align*}
		\rho^n_i \geq \frac14\,\rho^{k_i+1}_{j_i} \geq \frac{1}{96}\,\rho^{k_i} \geq \frac{1}{96}\cdot\frac{2}{ 3\lambda\log(1/\tau_n) }.
	\end{align*}
	This shows
	\begin{align*}
		\sum_{i=1}^n \rho^n_i \geq \frac{n}{2}\cdot \frac{1}{96}\cdot\frac{2}{ 3\lambda\log(1/\tau_n) }.
	\end{align*}
\end{proof}

%-------------------------------------------------------------------
% Main Deterministic Result
%-------------------------------------------------------------------
\subsection{Main Deterministic Result}
For the following simple fact we omit the proof.

\begin{lem}\label{lem:trivial}
	Let $C>0$, $\lambda\geq 1$, $n\geq 2$, and $0<\varepsilon\leq 1/n$ with
	\begin{align*}
		\frac{n}{\log\left(1/{\varepsilon}\right)}
			\leq C\cdot \left(\log(1/g(\varepsilon))\right)^4.
	\end{align*}
	Then there exists a constant $\tilde C>0$ that only depends on $C$ and $\lambda$ such that
	\begin{align*}
		g(\varepsilon) \leq {\tilde C}\cdot \exp\left(-{1}/{\tilde C}\cdot n^{1/5}\right).
	\end{align*}
\end{lem}

For $n\geq 4$, $\lambda\in{[1,\infty[}$, and $C>0$ we define
\begin{align}\label{eq:favorable}
	E_{C,n} = \left(
			F_{n}\cap J_{n}^+\cap J_{n}^-\cap G_{{1}/{2},n}\cap H_{C,n}
		\right)
		\cap
		\left(
			\bigcap_{k=\lceil n/2\rceil}^n \left(G_{1,k}\cap F_{k}\right)
		\right).
\end{align}
Note that $E_{C_1,n}\subseteq E_{C_2,n}$, if $0<C_1\leq C_2$.

\begin{cor}\label{cor:mainproperror}
	Let $\lambda\geq 1$, $n\geq 4$, $C>0$, and $f\in E_{C,n}$.
	Then there exists a constant $\tilde C>0$ that only depends on $C$ and $\lambda$ such that
	\begin{align*}
		\Delta_n \leq g(\tau_n)
			\leq {\tilde C}\cdot \exp\left(-{1}/{\tilde C}\cdot n^{1/5}\right).
	\end{align*}
\end{cor}
\begin{proof}
	The first inequality holds by definition of $G_{1,n}$. The
	second inequality is an immediate consequence of Corollary~\ref{cor:lowerbounddet},
	Proposition~\ref{prop:upperbounddet}, and Lemma~\ref{lem:trivial}.
\end{proof}

%-------------------------------------------------------------------
%-------------------------------------------------------------------
% Probabilistic Arguments
%-------------------------------------------------------------------
%-------------------------------------------------------------------
\section{Probabilistic Arguments}\label{sec:probabilistic}
In the previous section we studied the application of the
optimization algorithm to an element $f\in F$. In particular,
Corollary~\ref{cor:mainproperror} provides an exponentially small
error bound for functions $f$ belonging to subsets $E_{C,n}$ of $F$.
In this section we consider the special case of a Brownian motion
$W=(W(t))_{0\leq t\leq 1}$ and show that the probability of a Brownian
path belonging to $E_{C,n}$ tends to $1$ at an arbitrarily high
polynomial rate, see Corollary~\ref{cor:mainlowerboundprop}.
It turns out that this probability bound depends on the parameter $\lambda$.

Let us stress that all quantities defined in Section~\ref{sec:algo}
(e.g.,~$M_n,M,\tau_n,\ldots$) are now understood to depend on $W$ instead of $f$.
Hence these quantities are random. Furthermore, for a set of functions $A\subseteq F$
(e.g., $F_n,G_{1/2,n},\ldots$) we write $\prob(A)$ instead of $\prob(W\in A)$.

%-------------------------------------------------------------------
% Lower Bound for F_n
%-------------------------------------------------------------------
\subsection{Lower Bound for \texorpdfstring{$\prob\left(F_{n} \right)$}{P(Fn)}}
The following basic result is well-known, for completeness we add a proof.

\begin{lem}\label{lem:max-normal}
	Let $n\in\N$ and $Z_1,\ldots,Z_n$ be identically distributed with $Z_1\sim\ndist{0}{1}$.
	Then we have
	\begin{align*}
		\prob\left( \max_{1\leq k\leq n} \abs{Z_k} \leq t \right)
			\geq 1 - n\cdot\frac{2}{\sqrt{2\pi}}\cdot \frac{1}{t}\cdot\exp(-t^2/2)
	\end{align*}
	for all $t>0$.
\end{lem}
\begin{proof}
	We clearly have
	\begin{align*}
		\prob\left( \max_{1\leq k\leq n} \abs{Z_k} > t \right)
				\leq \sum_{k=1}^n \prob( \abs{Z_k} > t )
			= n\cdot \prob( \abs{Z_1} > t )
	\end{align*}
	for $t\in\R$. Combining this with the inequality
	\begin{align*}
		\prob( \abs{Z_1} > t )
			\leq \frac{2}{\sqrt{2\pi}}\cdot \frac{1}{t}\cdot\exp(-t^2/2)
	\end{align*}
	for $t>0$ yields the claim.
\end{proof}

\begin{lem}\label{lem:bound-F}
	We have
	\begin{align*}
	\prob\left(F_{n} \right) \geq 1 -7 \cdot n^{1-\lambda/72}
	\end{align*}
	for all $\lambda\geq1$ and for all $n\geq2$.
\end{lem}
\begin{proof}
	For $n\geq 2$ we denote by $i_n\in\{1,\dots,n\}$ the index of the interval which will
	be split in step $n$, i.e., with $\rho^n_{i_n} = \rho^n$ (note that $i_n$ is random).
	Moreover, we set $i_1=1$. Note that
	\begin{align*}
		t^{n+1}_{i_n-1}=t^n_{i_n-1}, \qquad
		t_{i_n+1}^{n+1}=t_{i_n}^n, \qquad
		t_{i_n}^{n+1}=t_{n+1}=(t^n_{i_n-1}+t_{i_n}^n)/2,
	\end{align*}
	and
	\begin{align*}
		t_{i_n}^{n+1}-t_{i_n-1}^{n+1}=t_{i_n+1}^{n+1}-t_{i_n}^{n+1}=(t_{i_n}^{n}-t_{i_n-1}^{n})/2.
	\end{align*}
	Furthermore, we define $X_1=\frac{W(t_1^1)-W(t_0^1)}{\sqrt{t_1^1-t_0^1}}=W(1)$ and
	\begin{align*}
		X_{2n} = \frac{ W(t_{i_n}^{n+1})-W(t_{i_n-1}^{n+1}) }
				{ \sqrt{t_{i_n}^{n+1}-t_{i_n-1}^{n+1}} }, \qquad
			X_{2n+1} = \frac{ W(t_{i_n+1}^{n+1})-W(t_{i_n}^{n+1}) }
				{ \sqrt{t_{i_n+1}^{n+1}-t_{i_n}^{n+1}} }
	\end{align*}
	for $n\geq1$. Let us stress that that
	\begin{align*}
		\{ X_1,\dots,X_{2n-1} \} =
			\left\{
				\frac{ W(t^k_i)-W(t^k_{i-1}) }
				{ \sqrt{t^k_i-t^k_{i-1}} }
				:\,1\leq i\leq k\leq n
			\right\}
	\end{align*}
	for $n\geq1$ and thus
	\begin{align}\label{eq:set-equality}
		\max_{1\leq k\leq 2n-1}\abs{ X_k }
			= \max_{1\le k\le n} \max_{1\le i\le k}
				\frac{ \abs{W(t^k_i)-W(t^k_{i-1})} }
				{ \sqrt{t^k_i-t^k_{i-1}} }.
	\end{align}
	For $n\geq 1$ we define
	\begin{align*}
		Y_n = \frac{ W(t_{i_n}^{n})-W(t_{i_n-1}^{n}) }
				{ \sqrt{t_{i_n}^{n}-t_{i_n-1}^{n}} }
			= \frac{ X_{2n}+X_{2n+1} }{ \sqrt{2} }.
	\end{align*}
	Note that for every $n\geq 1$ there exists a random index $j_n\in\{2(n-1),2(n-1)+1\}$
	with $Y_n=X_{j_n}$ where we use the convention $X_0=X_1$. This yields
	\begin{align}\label{eq:bound-Yn}
		\vert Y_n\vert \leq \max \left(
				\abs{ X_{2(n-1)} }, \abs{ X_{2(n-1)+1} }
			\right).
	\end{align}
	Finally, we define
	\begin{align*}
		Z_n = \frac{ X_{2n}-X_{2n+1} }{ \sqrt{2} }
	\end{align*}
	for $n\geq 1$. Since
	\begin{align*}
		Z_n = \frac{
				2W(t_{i_n}^{n+1})-\left( W(t_{i_n-1}^{n+1})+W(t_{i_n+1}^{n+1}) \right)
			}{
				\sqrt{{t_{i_n}^n-t_{i_n-1}^n}}
			}
			= \frac{
				W\left(\frac{t_{i_n}^n+t_{i_n-1}^n}{2}\right)-\frac{W(t_{i_n-1}^n)+W(t_{i_n}^n)}{2}
			}{
				\sqrt{(t_{i_n}^n-t_{i_n-1}^n)/4}
			},
	\end{align*}
	we have $Z_n\sim\ndist{0}{1}$. Furthermore, note that
	\begin{align*}
		X_{2n} = \frac{Y_n+Z_n}{\sqrt{2}}, \qquad
			X_{2n+1} = \frac{Y_n-Z_n}{\sqrt{2}}.
	\end{align*}
	Hence we get
	\begin{align*}
		\max\left( \abs{ X_{2n} }, \abs{ X_{2n+1} } \right)
			\leq \frac{ \abs{Y_n}+\abs{Z_n} }{ \sqrt{2} }
			\leq \frac{
				\max\left(\abs{ X_{2(n-1)} }, \abs{ X_{2(n-1)+1} }\right)+\abs{Z_n}
			}{
				\sqrt{2}
			}
	\end{align*}
	for $n\geq 1$ due to \eqref{eq:bound-Yn}. Combining this with the inequality
	\begin{align*}
		\frac{\frac{a}{\sqrt{2}-1}+b}{\sqrt{2}}
			\leq \frac{\max(a,b)}{\sqrt{2}-1}
	\end{align*}
	for $a,b\in\R$, we obtain by induction
	\begin{align}\label{eq:bound-Xn}
		\max_{1\leq k\leq 2n-1} \abs{ X_k }
			\leq \frac{
				\max\left( \abs{X_1},\abs{Z_1},\ldots,\abs{Z_{n-1}} \right)
			}{
				\sqrt{2}-1
			}
	\end{align}
	for all $n\geq 1$. Finally, combining \eqref{eq:set-equality},
	\eqref{eq:bound-Xn}, and Lemma~\ref{lem:max-normal} yields
	\begin{align*}
		\prob\left(F_{n}\right)
			&= \prob\left(
				\max_{1\leq k\leq 2n-1} \abs{ X_k } \leq \sqrt{\lambda\log(n)/4}
			\right) \\
		&\geq \prob\left(
				\max\left(|X_1|,|Z_1|,\dots,|Z_{n-1}|\right)
				\leq (\sqrt{2}-1)\cdot\sqrt{\lambda\log(n)/4}
			\right) \\
		&\geq \prob\left(
				\max\left(|X_1|,|Z_1|,\dots,|Z_{n-1}|\right)
				\leq \sqrt{\lambda\log(n)/36}
			\right) \\
		&\geq 1 - n\cdot\frac{2}{\sqrt{2\pi}}\cdot{8} \cdot n^{-\lambda/72} \\
		&\geq 1 -{7} \cdot n^{1-\lambda/72}
	\end{align*}
	for $\lambda\geq1$ and $n\geq2$.
\end{proof}

\begin{rem}\label{rem:distx}
	Let us comment on the distribution of the random variables $X_1,X_2,\ldots$
	defined in the proof of Lemma~\ref{lem:bound-F}.
	Obviously, the random variables $X_1,X_2,X_3$ are standard normally
	distributed and jointly Gaussian. In contrast to that,
	the random variables $X_1,X_2,X_3,X_4,X_5$ are not jointly Gaussian,
	but still $X_4$ and $X_5$ are standard normally distributed.
	However, computer simulations strongly suggest that $X_6$
	is not standard normally distributed.
	Since the evaluation points $t_0,t_1,\ldots$ are computed adaptively,
	we conjecture that $X_n$ is not standard normally distributed for all $n\geq 6$.
\end{rem}

%-------------------------------------------------------------------
% Lower Bound for G_{1/2,n}
%-------------------------------------------------------------------
\subsection{Lower Bound for \texorpdfstring{$\prob\left(G_{{1}/{2},n}\right)$}{P(G1/2,n)}
	and \texorpdfstring{$\prob\left(G_{{1},n} \right)$}{P(G1,n)}}

\begin{lem}\label{lem:probability-boundprop}
	We have
	\begin{align*}
		\prob\left(G_{{1},n} \right)
			\geq \prob\left(G_{{1}/{2},n} \right)
			\geq 1-8\cdot n^{1-\lambda/72}
	\end{align*}
	for all $\lambda\geq1$ and for all $n\geq2$.
\end{lem}
\begin{proof}
	For $n\in\N$ we denote by
	\begin{align*}
		\mathfrak{A}_n = \sigma(W(t_1),\ldots,W(t_n)) = \sigma(t_1,W(t_1),\ldots,t_n,W(t_n))
	\end{align*}
	the $\sigma$-algebra generated by $(W(t_1),\ldots,W(t_n))$.
	Note that $1_{F_{n}}$ is measurable w.r.t.\ $\mathfrak{A}_n$ for all $\lambda\geq1$ and $n\geq2$.

	Conditional on $\mathfrak{A}_n$, the minimizers over all subintervals $[t^n_{0},t^n_1],\ldots,[t^n_{n-1},t^n_{n}]$
	are independent with distribution (independent Brownian bridges)
	\begin{align*}
		\prob\left( \min_{t^n_{i-1}\le s\le t^n_i}W(s) < y \right)
			= \exp \left(
					-\frac{2}{t^n_i-t^n_{i-1}} \left(W(t^n_{i-1})-y\right) \left(W(t^n_{i})-y\right)
				\right)
	\end{align*}
	for $y<\min(W(t^n_{i-1}),W(t^n_i))$, see \cite[IV.4, p.~67]{handbook-bm} or \cite{Shepp}.
	For $\beta\in[0,1]$, we hence get
	\begin{align*}
		\prob ( &\Delta_n \leq \beta g(\tau_n) \cond \mathfrak{A}_n ) \\
		&= \prod_{i=1}^n \left(
				1-\exp \left(
					-\frac{2}{t^n_i-t^n_{i-1}}
					\left( W(t^n_{i-1})-M_n+\beta g(\tau_n) \right)
					\left( W(t^n_{i})-M_n+\beta g(\tau_n) \right)
				\right)
			\right) \\
		&\geq \prod_{i=1}^n \left( 1-\exp\left(-2\beta^2/\rho^n_i\right) \right) \\
		&\geq \left( 1-\exp\left(-2\beta^2/\rho^n\right) \right)^n.
	\end{align*}
	Then, Lemma~\ref{lem:rho} implies
	\begin{align*}
		\prob ( \Delta_n \leq \beta g(\tau_n) \cond \mathfrak{A}_n )
			&\geq \left( 1-\exp \left(-\beta^2\lambda\log(n)\right) \right)^n \\
		&\geq 1-n^{1-\beta^2\lambda}
	\end{align*}
	on $F_{n}$.
	Setting $B_n=\{\Delta_n \leq \beta g(\tau_n)\}$, we thus obtain
	\begin{align*}
		\prob(B_n) \geq \E{ \E{ 1_{B_n\cap F_{n}}\cond \mathfrak{A}_n } }
			= \E{ 1_{F_{n}} \cdot \E{ 1_{B_n}\cond \mathfrak{A}_n } }
			\geq \prob(F_{n}) \cdot \left( 1-n^{1-\beta^2\lambda} \right).
	\end{align*}
	Set $\beta=1/2$.
	Finally, Lemma~\ref{lem:bound-F} shows
	\begin{align*}
		\prob(B_n)
		\geq
		\left(
		1-7\cdot n^{1-\lambda/72}
		\right)^+
		\cdot
		\left(
		1-n^{1-\lambda/4}
		\right)^+
		\geq
		1- 8\cdot n^{1-\lambda/72}
	\end{align*}
	for $\lambda\geq1$ and $n\geq2$.
\end{proof}

%-------------------------------------------------------------------
% Lower Bound for J_{n}^+ and J_{n}^-
%-------------------------------------------------------------------
\subsection{Lower Bound for \texorpdfstring{$\prob\left(J_{n}^+ \right)$}{P(Jn+)}
	and \texorpdfstring{$\prob\left(J_{n}^- \right)$}{P(Jn+)}}

\begin{lem}\label{lem:probability-bound}
	We have
	\begin{align*}
		\prob(J_{n}^+) = \prob(J_{n}^-) \geq 1-n^{1-\lambda/8}
	\end{align*}
	for all $\lambda\geq1$ and for all $n\geq2$.
\end{lem}
\begin{proof}
	Let
	\begin{align*}
		Y_i = \max_{t^n_{i-1}\le s\le t^n_i}
			\frac{W(s)-L_n(s)}{\sqrt{t^n_i-t^n_{i-1}}},
	\end{align*}
	which is the maximum of a standard Brownian bridge, and thus
	$\prob(Y_i>y)=\exp(-2y^2)$ for $y\geq0$,
	see \cite[IV.4, p.~67]{handbook-bm} or \cite{Shepp}.
	Moreover, the family
	$\left(Y_1,\ldots,Y_n\right)$ is independent and so
	\begin{align*}
		\prob\left( \max_{1\leq i\leq n}Y_i \leq \frac 12\sqrt{\lambda\log(n)/4} \right)
			= \left( 1-\exp(-\lambda\log(n)/8) \right)^n
			\geq 1-n^{1-\lambda/8}.
	\end{align*}
	By symmetry, we obtain the same bound for $J_{n}^-$.
\end{proof}

%-------------------------------------------------------------------
% Lower Bound for H_{C,n}
%-------------------------------------------------------------------
\subsection{Lower Bound for \texorpdfstring{$\prob\left(H_{C,n}\right)$}{P(HC,n)}}\label{sec:LBH}
For $T>0$ and $z\geq 0$ let $R^{T,z}=(R^{T,z}(t))_{0\leq t\leq T}$ denote
a $3$-dimensional Bessel bridge from $0$ to $z$ on $[0,T]$, that is
a $3$-dimensional Bessel process started at $0$ conditioned
to have value $z$ at time $T$. In other words,
for independent Brownian bridges $B^T_1$, $B^T_2$, and $B^T_3$
from $0$ to $0$ on $[0,T]$, see, e.g., \cite[p.~274]{Pitman}, we have
\begin{align}\label{e1}
	(R^{T,z}(t))_{0\leq t\leq T} \eqdist \left(\sqrt{
			\left( \frac{z\cdot t}{T}+B^T_1(t) \right)^2
			+\left( B^T_2(t) \right)^2
			+\left( B^T_3(t) \right)^2
		}\right)_{0\leq t\leq T},
\end{align}
where $\eqdist$ denotes equality in distribution.
A consequence of \eqref{e1} is the following scaling property
\begin{align}\label{e2}
	(R^{T,z}(t))_{0\leq t\leq T} \eqdist
		\left(\frac{1}{\sqrt{c}}\,R^{c\cdot T,\sqrt{c}\cdot z}(c\cdot t)\right)_{0\leq t\leq T}
\end{align}
for all $c>0$. Moreover, if $z_1\leq z_2$, there exist $3$-dimensional Bessel bridges
$R^{T,z_1}$ and $R^{T,z_2}$ (on a common probability space) such that
\begin{align}\label{e3}
	R^{T,z_1}(t) \leq R^{T,z_2}(t)
\end{align}
for all $0\leq t\leq T$. We refer to \cite[Chap.~XI]{Revuz-Yor} for a detailed discussion
of Bessel processes and Bessel bridges.

\begin{lem}\label{lem1}
	For all $r\geq1$ there exists a constant $C>0$ such that for all $0<T\leq 1$ and $z\geq 0$ we have
	\begin{align*}
		\prob \left(
				\int_0^T \frac{1}{ (R^{T,z}(t))^2+\eps } \dt
				\geq C\cdot \left( \log(1/\eps) \right)^4
			\right)
			\leq C\cdot\eps^r
	\end{align*}
	for all $\eps>0$.
\end{lem}
\begin{proof}
	We may assume $T=1$ and $z=0$ due to \eqref{e2} and \eqref{e3}, respectively. In this case
	\begin{align*}
		R^{1,0} = B^{\text{ex}} = (B^{\text{ex}}(t))_{0\leq t\leq 1}
	\end{align*}
	is a Brownian excursion of length $1$, see, e.g., \cite[Lem.~15]{Pitman} or \cite{Williams}.

	Let $B=(B(t))_{t\geq 0}$ be a Brownian motion. Consider the stochastic differential equation
	\begin{align*}
		\mathrm d X(t) = \left( 4-\frac{(X(t))^2}{ 1-\int_0^t X(s)\ds } \right) \dt + 2\sqrt{X(t)}\,\mathrm d B(t),
			\qquad X(0)=0,
	\end{align*}
	to be solved on $[0,V(X)[$, where $V(X)=\inf\{t\geq 0:\,\int_0^t X(s)\ds=1\}$ and $X(t)=0$ for $t\geq V(X)$.
	For properties of this SDE and its solution, see \cite{Pitman}. In particular, there it is shown that
	this SDE has a unique continuous nonnegative strong solution $X=(X(t))_{t\geq 0}$.
	Moreover, this solution satisfies
	\begin{align}\label{e4}
		(L_1(x))_{x\geq 0} \eqdist (X(t))_{t\geq 0},
	\end{align}
	where $L_1=(L_1(x))_{x\in\R}$ denotes the local time of $B^{\text{ex}}$ up to time $t=1$.
	More precisely, $L_1$ is the continuous density with respect to the Lebesgue measure on $\R$
	of the push-forward measure of the Lebesgue measure on $[0,1]$ under the mapping $B^{\text{ex}}$, i.e.,
	\begin{align}\label{eq:localtime}
		\int_0^1 h(B^{\text{ex}}(t))\dt
			= \int_{-\infty}^\infty h(x)\cdot L_1(x)\dx
	\end{align}
	for all nonnegative Borel measurable $h\colon\R\to\R$.
	Note that $L_1(x)=0$ for $x\leq 0$.

	Consider the stochastic differential equation
	\begin{align*}
		\mathrm dY(t) =4\dt + 2\sqrt{Y(t)}\,\mathrm dB(t), \qquad Y(0)=0.
	\end{align*}
	It is known that this SDE has a unique strong solution $Y=(Y(t))_{t\geq0}$,
	which is a $4$-dimensional squared Bessel process started at $0$, i.e.,
	\begin{align}\label{eq:bessel-4}
		\left( Y(t) \right)_{t\geq0} \eqdist \left( \sum_{k=1}^4 \left(W_k(t)\right)^2 \right)_{t\geq0}
	\end{align}
	with independent Brownian motions $W_k=(W_k(t))_{t\geq0}$ for $k=1,\dots,4$,
	see, e.g., \cite[Chap.~XI]{Revuz-Yor}.
	Using a slight modification of the comparison principle \cite[Prop.~V.2.18]{Karatzas} we obtain
	\begin{align}\label{e5}
		X(t)\leq Y(t)
	\end{align}
	for $t\geq0$. Combining \eqref{eq:localtime}, \eqref{e4}, \eqref{e5}, and \eqref{eq:bessel-4} yields
	\begin{align}\label{eq:Bex-ineq}
		\begin{aligned}[c]
			\int_0^1 \frac{1}{ (B^{\text{ex}}(t))^2+\eps } \dt
				&\leq 1 + \int_0^1 \frac{ 1\left(B^{\text{ex}}(t)\leq 1\right) }{ (B^{\text{ex}}(t))^2+\eps }\dt \\
			&= 1 + \int_{0}^1\frac{1}{x^2+\eps}\cdot L_1(x)\dx \\
			&\eqdist 1 + \int_{0}^1 \frac{1}{t^2+\eps}\cdot X(t)\dt \\
			&\leq 1 + \int_{0}^1 \frac{1}{t^2+\eps}\cdot Y(t)\dt \\
			&\eqdist 1 + \int_{0}^1 \left( \sum_{k=1}^4 \frac{ \left(W_k(t)\right)^2 }{ t^2+\eps } \right)\dt
		\end{aligned}
	\end{align}
	for all $\eps>0$.
	
	Consider the Gaussian random function $\xi=(\xi(t))_{0\leq t\leq 1}$ given by $\xi(0)=1$ and
	\begin{align*}
		\xi(t) = \frac{W(t)}{ \sqrt{t}\cdot\sqrt{\log\left(1+\frac{1}{t}\right)} }
	\end{align*}
	for $0<t\leq 1$, which is bounded due to the law of the iterated logarithm.
	Using \cite[Thm.~12.1]{Lifshits} we get the existence of constants $C>0$ and $a>1$ such that
	\begin{align}\label{eq:prob-xi}
		\prob\left( \sup_{0\leq t\leq 1} \abs{\xi(t)} > q \right)
			\leq \frac{1}{a^{q^2}}
	\end{align}
	for all $q\geq C$. A small computation shows that there exist positive constants $C_1,C_2>0$ such that
	\begin{align*}
		\int_0^1 \frac{ t\cdot\log\left(1+\frac{1}{t}\right) }{ t^2+\eps } \dt
			\leq C_1\cdot \left(\log(1/\eps)\right)^2
	\end{align*}
	for all $0<\eps<C_2$, and consequently
	\begin{align*}
		\int_0^1 \frac{ (W(t))^2 }{ t^2+\eps } \dt
		= \int_0^1 (\xi(t))^2\cdot \frac{ t\cdot\log\left(1+\frac{1}{t}\right) }{ t^2+\eps } \dt
			\leq C_1\cdot \left( \sup_{0\leq t\leq 1} \abs{\xi(t)} \right)^2
			\cdot \left(\log(1/\eps)\right)^2.
	\end{align*}
	Let $r\geq1$. Using \eqref{eq:prob-xi} with $q=\log(1/\eps)$,
	there exists a constant $\tilde C_2>0$ such that
	\begin{align*}
		\prob \left(
				\int_0^1 \frac{ (W(t))^2 }{ t^2+\eps } \dt
				\leq C_1\cdot \left( \log(1/\eps) \right)^4
			\right)
			\geq 1-\eps^r
	\end{align*}
	for all $0<\eps<\tilde C_2$, and hence
	\begin{align}\label{eq:int-W-ineq}
		\prob \left(
				\int_{0}^1 \left( \sum_{k=1}^4 \frac{ \left(W_k(t)\right)^2 }{ t^2+\eps } \right)\dt
				\leq 4C_1\cdot \left( \log(1/\eps) \right)^4
			\right)
			\geq 1-4\cdot\eps^r
	\end{align}
	for all $0<\eps<\tilde C_2$. Combining \eqref{eq:Bex-ineq} and \eqref{eq:int-W-ineq}
	completes the proof.
\end{proof}

\begin{lem}\label{lem2}
	For all $r\geq1$ there exists a constant $C>0$ such that
	\begin{align*}
		\prob \left(
				\int_0^1\frac{1}{ (W(t)-M)^2+\eps }\dt
				\geq C\cdot \left( \log(1/\eps)\right)^4
			\right)
			\leq C\cdot\eps^r
	\end{align*}
	for all $\eps>0$.
\end{lem}
\begin{proof}
	Let $t^\ast\in{]0,1[}$ be the (a.s.~unique) minimizer of the Brownian motion
	$W$, i.e., $M=W(t^\ast)$. Conditionally on $t^\ast=s$, $M=m$, and $W(1)=w$,
	the process
	\begin{align*}
		\left( W(s-t)-m \right)_{0\leq t\leq s}
	\end{align*}
	is a $3$-dimensional Bessel bridge from $(0,0)$ to $(s,-m)$, and the process
	\begin{align*}
		\left( W(s+t)-m \right)_{0\leq t\leq 1-s}
	\end{align*}
	is a $3$-dimensional Bessel bridge from $(0,0)$ to $(1-s, w-m)$,
	see, e.g., \cite[Prop.~2]{Asmussen}.
	Let $r\geq1$ and $C>0$ be according to Lemma~\ref{lem1}.
	Then we have
	\begin{align*}
		\prob &\left(
				\int_0^1 \frac{1}{ (W(t)-M)^2+\eps }\dt
				\geq C\cdot \left(\log(1/\eps)\right)^4
				\cond t^\ast=s,\ M=m,\ W(1)=w
			\right) \\
		&\leq \prob\left(
				\int_0^s \frac{1}{ (W(t)-m)^2+\eps }\dt
				\geq \frac{C}{2}\cdot \left(\log(1/\eps)\right)^4
				\cond t^\ast=s,\ M=m,\ W(1)=w
			\right) \\
		&\qquad+ \prob\left(
				\int_s^1 \frac{1}{ (W(t)-m)^2+\eps }\dt
				\geq \frac{C}{2}\cdot \left(\log(1/\eps)\right)^4
				\cond t^\ast=s,\ M=m,\ W(1)=w
			\right) \\
		&\leq \frac{C}{2}\cdot\eps^r
			+  \frac{C}{2}\cdot\eps^r
	\end{align*}
	by Lemma~\ref{lem1}, which establishes the claim.
\end{proof}

\begin{lem}\label{lem:upper-bound-4}
	For all $r\geq1$ and for all $\lambda\geq1$ there exists a constant $C>0$ such that
	\begin{align*}
		\prob\left( H_{C,n} \right) \geq 1-C\cdot n^{-r}
	\end{align*}
	for all $n\geq 2$.
\end{lem}
\begin{proof}
	Fix $\lambda\geq 1$. Since $\tau_n\in\mathcal{A}$ and $\tau_n\geq 1/2^{n-1}$,
	we get that $(g(\tau_n))^2$ takes at most $n$ different values, which we denote by $A_n$
	(note that $A_n$ depends on $\lambda$). Hence we get
	\begin{align*}
		\prob(H_{C,n}^c) &=
		\prob \left(
		\int_0^1 \frac{1}{ (W(t)-M+g(\tau_n))^2 }\dt
		\geq C\cdot\left( \log\left({1}/{g(\tau_n)}\right) \right)^4
			\right) \\
		&\le
		\prob \left(
				\int_0^1 \frac{1}{ (W(t)-M)^2+g(\tau_n)^2 }\dt
				\geq C\cdot\left( \log\left({1}/{g(\tau_n)}\right) \right)^4
			\right) \\
		&= \sum_{\eps\in A_n} \prob \left(
				\int_0^1 \frac{1}{ (W(t)-M)^2+\eps }\dt
				\geq \frac{C}{16}\cdot \left( \log(1/\eps) \right)^4,\ 
				(g(\tau_n))^2=\eps
			\right).
	\end{align*}
	Moreover, due to $\tau_n\leq 1/n$ we have
	\begin{align*}
		(g(\tau_n))^2 \leq \lambda\log(n)/n,
	\end{align*}
	and thus there exists $N\geq2$ such that $g(\tau_n)^2\leq 1/\sqrt{n}$ for all $n\geq N$.
	Now, let $r\geq1$ and $C/16>0$ be according to Lemma~\ref{lem2}. Then we get
	\begin{align*}
		\prob(H_{C,n}^c) \leq n\cdot \frac{C}{16}\cdot (1/\sqrt{n})^r
			= \frac{C}{16}\cdot n^{1-\frac{r}{2}}
	\end{align*}
	for all $n\geq N$.
\end{proof}

%-------------------------------------------------------------------
% Main probabilistic result
%-------------------------------------------------------------------
\subsection{Main Probabilistic Result}

\begin{cor}\label{cor:mainlowerboundprop}
	For all $r\geq 1$ and all $\lambda\geq 72\cdot(2+r)$
	there exists a constant $C>0$ such that
	\begin{align*}
		\prob\left( E_{C,n} \right) \geq 1-C\cdot{n^{-r}}
	\end{align*}
	for all $n\geq 4$.
\end{cor}
\begin{proof}
	Let $r\geq 1$, $\lambda\geq 72\cdot(2+r)\geq 1$, and
	$C>0$ be according to Lemma~\ref{lem:upper-bound-4}.
	Combining Lemma~\ref{lem:bound-F},
	Lemma~\ref{lem:probability-boundprop},
	Lemma~\ref{lem:probability-bound}, and
	Lemma~\ref{lem:upper-bound-4} yields
	\begin{align*}
	\prob(E_{C,n}^c) &\leq \left(
			\prob(F_{n}^c)
			+ \prob({J_{n}^+}^c)
			+ \prob({J_{n}^-}^c)
			+ \prob(G_{{1}/{2},n}^c)
			+ \prob(H_{C,n}^c)
		\right) \\
	&\qquad+ \left(
			\sum_{k=\lceil n/2\rceil}^n \prob(G_{1,k}^c)
			+ \sum_{k=\lceil n/2\rceil}^n \prob(F_{k}^c)
		\right) \\
	&\leq \left(
			7 n^{-r}
			+n^{-r}
			+n^{-r}
			+8 n^{-r}
			+Cn^{-r}
		\right) \\
	&\qquad+ \left(
			n\cdot 8(n/2)^{1-\lambda/72}
			+n\cdot 7(n/2)^{1-\lambda/72}
		\right)
	\end{align*}
	for all $n\geq 4$.
\end{proof}

%-------------------------------------------------------------------
%-------------------------------------------------------------------
% Proof of Theorem 1
%-------------------------------------------------------------------
%-------------------------------------------------------------------
\section{Proof of Theorem~\ref{thm:main}}\label{sec:mainProof}
Let $r\geq1$ and $p\geq1$. Moreover, we fix
\begin{align*}
	\lambda\geq 144\cdot(1+p\cdot r) = 72 \cdot(2+2pr).
\end{align*}
According to Corollary~\ref{cor:mainlowerboundprop} there exists
a constant $C>0$ such that
\begin{align*}
	\prob\left( E_{C,n} \right) \geq 1-C\cdot{n^{-2pr}}
\end{align*}
for all $n\geq 4$. Furthermore, due to Corollary~\ref{cor:mainproperror}
there exists a constant $\tilde C>0$ such that for all $n\geq 4$
\begin{align*}
	\Delta_n = \Delta_{n,\lambda}(W) \leq {\tilde C}\cdot \exp(-{1}/{\tilde C}\cdot n^{1/5})
\end{align*}
if $W\in E_{C,n}$. Noting that
\begin{align*}
	\Delta_n \leq -\inf_{0\leq t\leq 1} W(t) \eqdist \abs{Z}
\end{align*}
with $Z\sim\ndist{0}{1}$, we obtain using the Cauchy-Schwarz inequality
\begin{align*}
	\E{ \abs{\Delta_n}^p }
		&= \E{ \abs{\Delta_n}^p\cdot 1_{E_{C,n}} }
			+ \E{ \abs{\Delta_n}^p\cdot 1_{E_{C,n}^c}} \\
	&\leq {\tilde C}^p\cdot \exp(-p/\tilde C\cdot n^{1/5})
		+ \sqrt{ \E{ \abs{\Delta_n}^{2p} } } \cdot \sqrt{1-\prob(E_{C,n})} \\
	&\leq {\tilde C}^p\cdot \exp(-p/\tilde C\cdot n^{1/5})
		+\sqrt{ \E{ \abs{Z}^{2p} } } \cdot \sqrt{C}\cdot n^{-pr},
\end{align*}
for all $n\geq 4$.
This completes the proof of Theorem \ref{thm:main}.

%-------------------------------------------------------
%-------------------------------------------------------
% Acknowledgement
%-------------------------------------------------------
%-------------------------------------------------------
\section*{Acknowledgement}
We thank Klaus Ritter for valuable discussions and comments.

%-------------------------------------------------------------------
%-------------------------------------------------------------------
% References
%-------------------------------------------------------------------
%-------------------------------------------------------------------
\bibliographystyle{plainnat}
\bibliography{references}

\end{document}